\documentclass[10pt,reqno]{article}

\usepackage[a4paper,margin=1.15in]{geometry}

\usepackage[english]{babel}
\usepackage{microtype}
\usepackage{titlesec}

\usepackage{amsmath, amsthm, amssymb, mathtools}
\usepackage{latexsym}

\usepackage[T1]{fontenc}
\usepackage{bbm}
\usepackage[cal=boondoxo]{mathalfa}

\usepackage{xspace}
\usepackage{setspace}
\usepackage{natbib}
\usepackage{tabularx}
\usepackage{stmaryrd}
\usepackage{dsfont}
\usepackage{euscript}
\usepackage{enumitem}
\usepackage{algpseudocode}
\usepackage{algorithm}

\usepackage{xcolor}
\usepackage[utf8]{inputenc}
\usepackage{xcolor}
\usepackage{makeidx}

\usepackage[bibliography=common]{apxproof}
\usepackage[normalem]{ulem}
\usepackage[
pdftex,
colorlinks=true,
linkcolor=blue!50!black,
citecolor=blue!50!black,
urlcolor=blue!50!black
]{hyperref}
\usepackage{cleveref}
\usepackage{authblk}
\newtheorem{theorem}{Theorem}[]
\newtheorem{lemma}[theorem]{Lemma}
\newtheorem{corollary}[theorem]{Corollary}
\newtheorem{proposition}[theorem]{Proposition}

\newtheorem{remark}{Remark}[]

\newcommand{\cH}{\mathcal{H}}

\newcommand{\cL}{\mathcal{L}}
\newcommand{\cM}{\mathcal{M}}

\newcommand{\cX}{\mathcal{X}}

\newcommand{\wt}{\widetilde}

\newcommand{\bbR}{\mathbb{R}}
\newcommand{\bbE}{\mathbb{E}}
\newcommand{\bbP}{\mathbb{P}}

\newcommand{\bbC}{\mathbb{C}}
\newcommand{\bbH}{\mathbb{H}}

\newcommand{\bfI}{\mathbf{I}}
\newcommand{\bfT}{\mathbf{T}}

\newcommand{\abs}[1]{| #1 |}

\newcommand{\Babs}[1]{\Big| #1 \Big|}

\newcommand{\norm}[1]{\| #1 \|}
\newcommand{\bnorm}[1]{\big\| #1 \big\|}

\newcommand{\LRnorm}[1]{\left\| #1 \right\|}

\newcommand{\la}{\langle}
\newcommand{\ra}{\rangle}
\newcommand{\Tr}{\operatorname{Tr}}
\newcommand{\HS}{\mathrm{HS}}

\newcommand{\Ran}{\mathrm{Ran}}

\newcommand{\spec}{\operatorname{spec}}

\definecolor{darkred}{RGB}{150,0,0}

\title{A short operator proof of Hoeffding inequalities\\for Markov chains}
\author[1]{Oleksii Kachaiev%
\thanks{MaLGa center, Dipartimento di Matematica, Università degli Studi di Genova, Genoa, Italy}}
\date{September 19, 2026}

\hypersetup{
  pdftitle={A short operator proof of Hoeffding inequalities for Markov chains},
  pdfauthor={Oleksii Kachaiev},
  pdfsubject={Concentration inequalities for Markov chains},
  pdfkeywords={
    Hoeffding inequality,
    Markov chains,
    concentration inequalities,
    absolute spectral gap,
    Markov operators,
    Leon--Perron operator,
    matrix concentration,
    random matrices
  },
  pdfdisplaydoctitle=true
}

\begin{document}

\maketitle

\begin{abstract}
We give a short operator-theoretic proof of the sharp Hoeffding inequality for additive functionals of a Markov chain on a general state space with an \( L^2(\pi) \) spectral gap, recovering the bound of~\citet{fan2021hoeffding}. We further show that the same argument yields the Hoeffding inequality of~\citet{neeman2024concentration} for Markov-dependent random matrices. The proof rests on two key observations. First, the transition operator factors as \(P = Q_\lambda^{1/2} K Q_\lambda^{1/2}\) through the Le{\'o}n--Perron operator \(Q_\lambda\), with \(\norm{K} = 1\). This bounds the moment generating function by a product of one-step operator norms. Second, each one-step operator is a rank-one perturbation of a multiplication operator, so bounding its norm reduces to verifying a scalar resolvent condition. Scalar convexity and the classical Hoeffding lemma then give the desired estimate. This bypasses the asymptotic cumulant generating function, the essential-spectrum analysis and the extremal two-state comparison used in existing proofs~\citep{leon2004optimal,miasojedow2014hoeffding,fan2021hoeffding}. In the matrix setting, the multi-matrix Golden--Thompson inequality first reduces the moment generating function to an operator problem on a lifted Hilbert space. The same two steps then apply: the projection onto constants becomes finite-rank, and a Schur-complement argument yields the same scalar estimate. The Hoeffding inequality for Markov-dependent random matrices of~\citet{neeman2024concentration} follows. For complex Hermitian summands, working directly on the complex Hilbert space improves the bound by a factor of \(2\).
\end{abstract}

\section{Introduction}

Let \( (X_t)_{t\geq 1} \) be a stationary Markov chain on a general measurable state space \( \cX \), with transition kernel \(P\) and invariant probability measure \( \pi \). Given a bounded measurable function \(f:\cX\to\bbR\), consider the additive functional
\begin{equation}
S_n(f):=\sum_{t=1}^n f(X_t).
\end{equation}
Under the usual ergodicity assumptions, the Markov-chain ergodic theorem gives \( n^{-1}S_n(f)\longrightarrow \pi(f) \; \text{a.s.} \); see, e.g.,~\citet{meyn2012markov,douc2018markov}. We ask instead for a non-asymptotic version of this statement: how large can the deviation \( S_n(f)-n\pi(f) \) be at a fixed time \(n\)?

\medskip
The particular result we revisit is~\citet[Theorem~1]{fan2021hoeffding}: if the Markov chain \( (X_t)_{t\geq 1} \) admits an absolute spectral gap \( 1 - \lambda > 0 \) on \( L^2(\pi) \), then for arbitrary bounded \(f_t: \cX \to [a_t,b_t]\) and \(\varepsilon > 0\),
\begin{equation}
\label{eq:intro-fan-bound}
\bbP_\pi\left\{ \sum_{t=1}^n f_t(X_t)- \sum_{t=1}^n \pi(f_t) \geq\varepsilon \right\}
\leq \exp\left( -\frac{1-\lambda}{1+\lambda} \frac{2\varepsilon^2}{\sum_{t=1}^n(b_t-a_t)^2} \right).
\end{equation}
The bound has the same form as the classical Hoeffding inequality for independent bounded random variables~\citep{boucheron2013concentration}, up to the dependence penalty \( (1-\lambda)/(1+\lambda) \) in the exponent. In particular, when the chain is i.i.d.\ with law \(\pi\), we have \(\lambda=0\) and so~\eqref{eq:intro-fan-bound} reduces exactly to the classical Hoeffding bound.

\medskip
The problem has been studied rather extensively. Following Hoeffding's original inequality~\citep{hoeffding1963probability} in the independent case, similar bounds for finite-state reversible Markov chains were developed using spectral methods by, among others, \citet{gillman1993hidden,dinwoodie1995probability}. A different extension under uniform ergodicity was obtained by~\citet{glynn2002hoeffding}. The sharp spectral result for finite reversible chains is due to~\citet{leon2004optimal}. Their argument identifies a two-state chain with the same stationary mean and spectral parameter and shows, through convex majorization, that this chain is extremal. This two-state comparison subsequently became a central ingredient of the sharp spectral approach. \citet{miasojedow2014hoeffding} reinterpreted the Le{\'o}n--Perron argument as an \(L^2(\pi)\) operator framework under a suitable spectral gap assumption to extend the result to Markov chains on a general state space. Related finite-state Chernoff--Hoeffding bounds for nonreversible chains were derived by~\citet{chung2012chernoff}.  Two further lines of work are close to the present one. \citet{paulin2015concentration} proves a version of McDiarmid's bounded-differences inequality for Markov chains with constants proportional to the mixing time, together with variance bounds and Bernstein-type inequalities for sums, and introduces the \emph{pseudo spectral gap},
which plays for nonreversible chains the role that the spectral gap plays for reversible ones.  \citet{rao2019hoeffding} gives a short proof of a Hoeffding-type inequality on a finite state space, for time-dependent observables and without assuming reversibility. The proof bounds even moments of the sum rather than using a two-state comparison; we return to that argument in Section~\ref{sec:comparison}. Finally, \citet{fan2021hoeffding} obtained the general-state-space result~\eqref{eq:intro-fan-bound} for arbitrary time-dependent bounded observables. For a single time-independent observable, they proposed replacing the absolute spectral parameter by the positive part of the right spectral endpoint, which recovers the sharp Le{\'o}n--Perron bound in the reversible case. Other results both on Hoeffding-type and more general exponential concentration inequalities for Markov chains have also been derived under different sets of assumptions; see, among others, \citet{kontoyiannis2005large,adamczak2015exponential}.

\medskip
The present note does not improve the scalar Hoeffding inequality or its dependence on the spectral gap; Theorem~\ref{thm:scalar-hoeffding} below is~\citet[Theorem~1]{fan2021hoeffding}. What we contribute is a shorter and, we believe, more explanatory proof. Moreover, the same proof yields the matrix Hoeffding inequality of~\citet{neeman2024concentration} and, in the complex Hermitian case, improves it (Corollary~\ref{cor:complex-hermitian}). The proof of~\citet{fan2021hoeffding}, following the operator-theoretic framework established in the earlier literature, first bounds the Markov-chain moment generating function (MGF) by the norm of a certain operator and subsequently analyzes this norm through an asymptotic cumulant generating function (CGF), finite-valued approximation, essential-spectrum arguments, convex comparison with a two-state chain, and an explicit two-dimensional eigenvalue calculation followed by a convexity reduction to the classical Hoeffding lemma. Our observation is that the \emph{convexity argument} that appears near the end of that proof can instead be applied \emph{immediately} after the MGF has been reduced to the operator norm bound; a detailed comparison is given in Section~\ref{sec:comparison}.

\medskip
The reason this rearrangement is possible is structural, and it is worth stating separately. Lemma~\ref{lem:op-bound} reduces the Markov-chain MGF to a product of one-step operator norms, and all the information about \( f_t \) is carried by these norms alone. Bounding one such norm is, by Proposition~\ref{prop:secular}, exactly a one-dimensional problem: the norm is determined by a scalar equation, so any scalar exponential-moment estimate for its root transfers to the Markov setting through the same two steps. Hoeffding's classical lemma gives Theorem~\ref{thm:scalar-hoeffding}.

\medskip
Finally, we show that this proof architecture extends easily to the noncommutative setting of random matrices considered by~\citet{neeman2024concentration}. Once the multi-matrix Golden--Thompson inequality has reduced the matrix MGF to an operator problem on a lifted Hilbert space, the same argument applies: the projection onto constants becomes finite-rank rather than rank-one, and a Schur-complement argument leads to the operator-valued analogue for a self-adjoint matrix-valued observable \(T\). Since the spectrum of \(T\) lies in a scalar interval, the same endpoint inequality lifts directly by functional calculus. No operator convexity is needed. Thus, once the noncommutative Golden--Thompson reduction has been performed, the remainder of the matrix Hoeffding argument follows the scalar proof almost verbatim. Because the argument runs on the complex Hilbert space throughout, it treats complex Hermitian observables directly, without the realification step of~\citet{neeman2024concentration} and the factor of \( 2 \) it incurs; see Corollary~\ref{cor:complex-hermitian}.

\section{Scalar Hoeffding inequality}
\label{sec:scalar}

\paragraph{Notation.}
For an integrable function \(h:\cX\to\bbC\), write \( \pi(h):=\int_{\cX}h(x)\,\pi(dx) \). Throughout, \(L^2(\pi)\) denotes the complex Hilbert space
\(L^2(\cX,\pi;\bbC)\), equipped with \( \langle g,h\rangle :=\int_{\cX}\overline{g(x)}h(x)\,\pi(dx) \) and \( \norm{h}_\pi^2 :=\sqrt{\la h, h \ra} \). Let \( L_0^2(\pi) := \{h\in L^2(\pi):\pi(h)=0\} \) denote its closed subspace of mean-zero functions. For a bounded operator \(A\) on \(L^2(\pi)\), let \(\norm{A}\) denote its operator norm. For a bounded \( f \), let \( M_f: L^2(\pi) \to L^2(\pi) \) denote the multiplication operator, i.e., for \( h \in L^2(\pi) \), \( (M_f h)(x) = f(x)h(x) \). For \( u, v \in L^2(\pi) \), the operator \( u \otimes v\) acts as \( (u\otimes v) h = \la v, h \ra u \). For a bounded operator \( A \), \( \spec(A) \) denotes its spectrum, i.e.\ the set of \( z \in \bbC \) such that \( A - z I \) has no bounded inverse. Let \( \mathbf 1 \) denote the constant function \( x \mapsto 1 \), while \( e^f \) stands for \( x \mapsto e^{f(x)} \). The notation \( \bbP_\pi \) and \( \bbE_\pi \)  is used to highlight stationarity, i.e., \( X_1 \sim \pi \). All pointwise bounds on the observables below are understood \( \pi \)-a.e.; this suffices because \( \norm{M_g} \) is the \( \pi \)-essential supremum of \( \abs{g} \) and \( X_t \sim \pi \) for every \( t \) under \( \bbP_\pi \).

\paragraph{Setup.} We view the Markov transition kernel \(P\) as the bounded operator on \(L^2(\pi)\) defined by
\begin{equation}
(Ph)(x):=\int_{\cX}h(x')\,P(x,dx'),
\qquad h\in L^2(\pi).
\end{equation}
This operator is the complexification of its restriction to \(L^2(\cX,\pi;\bbR)\), and its operator norm is unchanged under complexification. Let \( \Pi: L^2(\pi) \to L^2(\pi) \) be the orthogonal projector onto the subspace of constant functions, i.e., \( \Pi h = \pi(h) \mathbf 1 \). Note that \( \Pi = \mathbf 1 \otimes \mathbf 1 \) and \( P \Pi = \Pi P = \Pi \), by \( P \mathbf 1 = \mathbf 1 \) and \( \pi P = \pi \) respectively. Assuming \( \norm{P - \Pi} =: \lambda < 1 \), we refer to \( 1 - \lambda \) as the \emph{absolute spectral gap}. Define \( Q_\lambda: L^2(\pi) \to L^2(\pi) \) by
\begin{equation}
Q_\lambda
:= \lambda I + (1-\lambda) \Pi
= \Pi + \lambda (I - \Pi),
\end{equation}
known as the Le{\'o}n--Perron operator~\citep{leon2004optimal}. Since \( (Q_\lambda h)(x) = \lambda h(x) + (1-\lambda) \pi(h) \), the operator \( Q_\lambda \) is itself the Markov operator of a chain: at each step such a chain stays at its current state with probability \( \lambda \) and draws a fresh sample from \( \pi \) with probability \( 1 - \lambda \). That chain is reversible with respect to \( \pi \) and satisfies \( \norm{Q_\lambda - \Pi} = \lambda \), so it is the \emph{laziest} chain with the prescribed spectral parameter.
Note that \( Q_\lambda \) is self-adjoint and \( Q_\lambda \succeq \lambda I \succeq 0 \), since \( Q_\lambda - \lambda I = (1-\lambda) \Pi \); hence \( Q_\lambda^{1/2} \) is well defined, and \( Q_\lambda^{1/2} = \Pi + \sqrt{\lambda}\,(I - \Pi) \) because \( \Pi \) and \( I - \Pi \) are orthogonal projections.

\medskip
We can now state the main result. Its proof, given after a few remarks, relies on two lemmas, which we prove in the remainder of this section.

\begin{theorem}
\label{thm:scalar-hoeffding}
Let \( (X_t)_{t \ge 1} \) be a stationary Markov chain taking values in a measurable space \( \cX \) with invariant measure \( \pi \) and absolute spectral gap \( 1 - \lambda > 0 \). Then, for every sequence of bounded measurable functions \(f_t: \cX \to [a_t,b_t]\) satisfying \(b_t>a_t\), and every \(\varepsilon>0\), the following bound holds:
\[
\bbP_\pi\left\{
\sum_{t=1}^n f_t(X_t)- \sum_{t=1}^n \pi(f_t) \geq \varepsilon
\right\}
\leq \exp\left( -\frac{1-\lambda}{1+\lambda} \cdot \frac{\varepsilon^2}{2\sum_{t=1}^n(b_t-a_t)^2/4} \right).
\]
\end{theorem}
The corresponding lower-tail bound follows by applying the theorem to \(-f_t\), whose range is \([-b_t,-a_t]\) and hence has the same width \(b_t-a_t\); a two-sided bound follows by a union bound.

\begin{remark}
\label{rem:monotone-in-lambda}
Theorem~\ref{thm:scalar-hoeffding} remains valid with \( \lambda \) replaced by any \( \lambda' \in [\norm{P-\Pi}, 1) \). Indeed, the factorization \( P = Q_{\lambda'}^{1/2} K' Q_{\lambda'}^{1/2} \) with \( K' := \Pi + \lambda'^{-1}(P-\Pi) \) holds by construction, and the only further property used in Lemma~\ref{lem:op-bound} is \( \norm{K'} \le 1 \), which holds for every such \( \lambda' \). This is what one needs in practice, where an upper bound on \( \norm{P - \Pi} \) is available rather than its exact value.
\end{remark}

\begin{remark}
\label{rem:nonstationary}
For an initial law \( \mu \ll \pi \), the path law under \( \bbP_\mu \) has density \( \tfrac{d\mu}{d\pi}(X_1) \) with respect to the path law under \( \bbP_\pi \), so \( \bbP_\mu(A) \leq \bnorm{\tfrac{d\mu}{d\pi}}_\pi \bbP_\pi(A)^{1/2} \) by Cauchy--Schwarz, and Theorem~\ref{thm:scalar-hoeffding} holds under \( \bbP_\mu \) with the prefactor \( \norm{d\mu/d\pi}_\pi \) and half the exponent; if \( d\mu/d\pi \in L^\infty(\pi) \), H\"older gives the prefactor \( \norm{d\mu/d\pi}_{L^\infty(\pi)} \) and leaves the exponent unchanged. We do not pursue this further; see~\citet{fan2021hoeffding} and~\citet{paulin2015concentration}.
\end{remark}

\begin{remark}
The special case of independent random variables, for which \( P(x, dx') = \pi(dx') \) and hence \( \lambda = 0 \), recovers the classical Hoeffding inequality.
\end{remark}

\medskip
Theorem~\ref{thm:scalar-hoeffding} itself is not new: it is precisely Theorem~1 of~\citet{fan2021hoeffding}, stated above as~\eqref{eq:intro-fan-bound}, and it contains the time-independent result of~\citet{miasojedow2014hoeffding}. What is new in the present note is \emph{the proof}: it is shorter and more direct than the existing arguments. We provide detailed comparison with the existing arguments later in Section~\ref{sec:comparison}.

\medskip
We call the bound \emph{sharp} in the sense of~\citet{leon2004optimal}, and it is worth being precise about what that means. For a finite reversible chain with second largest eigenvalue \( \lambda \) and \( f \) valued in \( [0,1] \) with \( \pi(f) = \mu \), their Theorem~1 bounds \( \bbP_\pi[S_n \ge n(\mu+\varepsilon)] \) by \( \exp\{-n I(\mu+\varepsilon)\} \), where \( I \) is the exact large-deviation rate function of the two-state chain with transition matrix \( \lambda I + (1-\lambda) \mathbf 1 \mu^{\top} \); the Gaussian form then follows from \( I(x) \ge 2 \tfrac{1-\lambda_0}{1+\lambda_0}(x-\mu)^2 \) with \( \lambda_0 := \max(0,\lambda) \). The quadratic coefficient is optimal in the local sense. Indeed, for the two-state chain \(Q_\lambda\) with \(\mu=1/2\), \( I(1/2+\varepsilon)
= 2\frac{1-\lambda}{1+\lambda}\varepsilon^2 +o(\varepsilon^2) \) for \( \varepsilon \to 0 \).  Thus the factor \( (1-\lambda)/(1+\lambda) \) cannot be replaced by any larger one. The proof below establishes the inequality; sharpness is inherited from that example and is not reproved.

\begin{proof}
By the standard Chernoff argument, it is enough to provide uniform in \( \theta \in \bbR \) bound for MGF
\begin{equation}
\bbE_\pi \left[\exp\left(\theta \sum_{t=1}^n f_t(X_t) - \theta \sum_{t=1}^n \pi(f_t)\right)\right]
\leq \exp\left(\frac{1+\lambda}{1-\lambda} \cdot \frac{\theta^2}{2} \cdot \sum_{t=1}^n \frac{(b_t-a_t)^2}{4} \right).
\label{eq:mgf-target}
\end{equation}
We prove the bound in two steps. First, Lemma~\ref{lem:op-bound} demonstrates that
\begin{equation}
\label{eq:mgf-bound}
\bbE_\pi\exp\!\left(\theta\sum_{t=1}^{n}f_t(X_t)\right)
\leq \prod_{t=1}^n \bnorm{M_{e^{\theta f_t}}^{1/2} Q_\lambda M_{e^{\theta f_t}}^{1/2}}.
\end{equation}
Second, Lemma~\ref{lem:rank1-resolvent} shows that, for every measurable function \( f:\cX\to[a,b] \) and every \( \theta\in\bbR \), we have
\begin{equation}
\label{eq:op-norm-bound}
\bnorm{M_{e^{\theta f}}^{1/2} Q_\lambda M_{e^{\theta f}}^{1/2}}
\leq \exp\left( \theta \pi(f) + \frac{1 + \lambda}{1 - \lambda} \frac{\theta^2 (b -a)^2}{8} \right).
\end{equation}
Multiplying~\eqref{eq:mgf-bound} by \( \exp\left(-\theta \sum_{t} \pi(f_t)\right) \) and inserting~\eqref{eq:op-norm-bound} cancels the terms \( \theta \pi(f_t) \) and gives~\eqref{eq:mgf-target}; the Cram{\'e}r--Chernoff method~\citep{boucheron2013concentration} then recovers the claim.
\end{proof}

It remains to prove the lemmas used above.

\begin{lemma}[Le{\'o}n--Perron operator domination]
\label{lem:op-bound}
Under the setup of Theorem~\ref{thm:scalar-hoeffding}, for every \( \theta \in \bbR \),
\[
\bbE_\pi\exp\!\left(\theta\sum_{t=1}^{n}f_t(X_t)\right)
\leq \prod_{t=1}^n \bnorm{M_{e^{\theta f_t}}^{1/2} Q_\lambda M_{e^{\theta f_t}}^{1/2}}.
\]
\end{lemma}
\begin{proof}
For brevity, set \( M_t := M_{e^{\theta f_t}} \). Define the operator \( K: L^2(\pi) \to L^2(\pi) \) by
\begin{equation}
\label{eq:scalar-k-def}
K := \begin{cases}
\Pi + \lambda^{-1}(P-\Pi), & \lambda>0,\\
I, & \lambda=0.
\end{cases}
\end{equation}
Note that \( K \) is nonexpansive. Indeed, for \( \lambda > 0 \), the two terms act on the orthogonal subspaces \(\Ran(\Pi)\) and \(\Ran(I-\Pi)\), respectively, and hence
\begin{equation}
\norm{ K }
\leq \max \big\{ \norm{\Pi}, \lambda^{-1} \norm{ P - \Pi } \big\}
= 1.
\end{equation}
Moreover, \( K \mathbf 1 = \mathbf 1 \); therefore \( \norm{K} = 1 \). Finally, in both cases, \( P=Q_\lambda^{1/2}KQ_\lambda^{1/2} \). For \(\lambda>0\) this follows by construction, while for \(\lambda=0\), since \(P=\Pi\) and \(Q_0=\Pi\), we have \( Q_0^{1/2} K Q_0^{1/2} = \Pi =P \).

\medskip \noindent
Setting \(g_{n+1} = \mathbf 1\) and \(g_t:=M_tPg_{t+1}\), the Markov property gives, by backward induction,
\begin{equation}
g_t(x)
= \bbE\!\left[ \exp\!\left(\theta\sum_{s=t}^{n}f_s(X_s)\right) \middle| X_t=x \right].
\end{equation}
Since \(X_1 \sim \pi\), integrating \(g_1\) yields
\begin{equation}
\bbE_\pi\exp\!\left(\theta\sum_{t=1}^{n}f_t(X_t)\right)
=\la \mathbf 1,g_1 \ra
=\la \mathbf 1,M_1PM_2P \cdots PM_n\mathbf 1\ra.
\label{eq:feynman-kac}
\end{equation}
Define \( C_t := Q_\lambda^{1/2} M_t Q_\lambda^{1/2} \). Using \( Q_\lambda^{1/2} \mathbf 1 = \mathbf 1 \),
\begin{equation}
\label{eq:lemma6-inequality}
\la \mathbf 1, M_1 P M_2 P \cdots P M_n \mathbf 1 \ra
= \la \mathbf 1, C_1 K C_2 K \cdots K C_n \mathbf 1 \ra
\leq \prod_{t=1}^n \norm{C_t}, 
\end{equation}
Using \( \norm{A^\ast A}=\norm{AA^\ast} \),
\begin{equation}
\norm{C_t}
= \bnorm{ (M_t^{1/2} Q_\lambda^{1/2})^\ast M_t^{1/2} Q_\lambda^{1/2} }
= \bnorm{ M_t^{1/2} Q_\lambda^{1/2} (M_t^{1/2} Q_\lambda^{1/2})^\ast }
= \bnorm{M_t^{1/2} Q_\lambda M_t^{1/2}}
\end{equation}
Substituting this identity into~\eqref{eq:lemma6-inequality} proves the claim.
\end{proof}

The representation~\eqref{eq:feynman-kac} is the standard discrete-time Feynman--Kac formula with time-dependent potentials; see, e.g.,~\citet{fitzsimmons1999kac}. The conceptual point is the factorization \( P = Q_\lambda^{1/2} K Q_\lambda^{1/2} \) with \( \norm{K} = 1 \), which may be inserted throughout the Feynman--Kac product and makes the claim immediate.

\begin{lemma}[Rank-one resolvent bound]
\label{lem:rank1-resolvent}
Under the setup of Theorem~\ref{thm:scalar-hoeffding}, let \(a < b\), let \(f: \cX \to [a,b] \) be measurable, and let
\(\lambda \in [0,1)\). Then, for every \( \theta \in \bbR \),
\begin{equation}
\label{eq:r-theta-def}
\bnorm{M_{e^{\theta f}}^{1/2} Q_\lambda M_{e^{\theta f}}^{1/2}}
\leq \exp\left( \theta \pi(f) + \frac{1 + \lambda}{1 - \lambda} \frac{\theta^2 (b -a)^2}{8} \right)
=: R_\theta.
\end{equation}
\end{lemma}
\begin{proof}
Set \( u := e^{\theta f/2} \). Then
\begin{equation}
\label{eq:ort-decomposition-mqm}
M_{e^{\theta f}}^{1/2} Q_\lambda M_{e^{\theta f}}^{1/2}
= \lambda M_{e^{\theta f}} + (1-\lambda) \, M_{e^{\theta f}}^{1/2} ( \mathbf 1 \otimes \mathbf 1 ) M_{e^{\theta f}}^{1/2}
= \lambda M_{e^{\theta f}} + (1-\lambda) \, u \otimes u.
\end{equation}
Therefore, for any \( h \in L^2(\pi) \),
\begin{align}
\la h, M_{e^{\theta f}}^{1/2} Q_\lambda M_{e^{\theta f}}^{1/2} h \ra
&= \lambda \la h, M_{e^{\theta f}} h \ra + (1-\lambda) \la h, (u \otimes u) h \ra \\
&= \lambda \la h, M_{e^{\theta f}} h \ra + (1-\lambda) \abs{\la u, h \ra}^2.
\end{align}
Fix any \(r>0\) such that
\begin{equation}
\label{eq:r-admissible}
r - \lambda e^{\theta x}
> 0,
\qquad \forall x \in [a,b].
\end{equation}
Let \( w_\theta := r - \lambda e^{\theta f} \). Since \( f(x) \in [a,b] \), for every \( \theta \in \bbR \), \( w_\theta > 0 \). Therefore, Cauchy-Schwarz gives
\begin{equation}
\abs{\la u, h \ra}^2
\leq \LRnorm{\frac{u}{\sqrt{w_\theta}}}_\pi^2 \bnorm{\sqrt{w_\theta} \, h}_\pi^2
= \pi\!\left( \frac{e^{\theta f}}{r - \lambda e^{\theta f}} \right) \la h, (r I - \lambda M_{e^{\theta f}}) h \ra.
\end{equation}
Consequently,
\begin{equation}
\label{eq:fan-to-prove}
(1 - \lambda) \,  \pi\!\left( \frac{e^{\theta f}}{r - \lambda e^{\theta f}} \right)
\leq 1
\implies M_{e^{\theta f}}^{1/2}Q_\lambda M_{e^{\theta f}}^{1/2}
\preceq rI.
\end{equation}
For \(r\) satisfying~\eqref{eq:r-admissible}, define
\begin{equation}
\label{eq:psi-r-def}
\psi_r(x)
:= \frac{e^{\theta x}} {r-\lambda e^{\theta x}},
\qquad x\in[a,b].
\end{equation}
Condition~\eqref{eq:r-admissible} ensures that \(\psi_r\) is well defined on \([a,b]\). Moreover, \(\psi_r\) is convex on \([a,b]\) since
\begin{equation}
\psi_r''(x)
= \frac{\theta^2 re^{\theta x} (r + \lambda e^{\theta x})}{(r - \lambda e^{\theta x})^3}
\ge 0.
\end{equation}
Convexity gives the pointwise bound,
\begin{equation}
\label{eq:psi-r-fx-convexity}
\psi_r(f(x))
\leq \frac{b - f(x)}{b - a} \psi_r(a) + \frac{f(x) - a}{ b -a } \psi_r(b).
\end{equation}
Integrating this inequality gives,
\begin{equation}
\label{eq:target-after-convexity}
\pi\left( \frac{e^{\theta f}}{r - \lambda e^{\theta f}} \right)
= \pi(\psi_r \circ f)
\leq (1 - \mu) \psi_r(a) + \mu \psi_r(b),
\qquad \text{where} \quad
\mu := \frac{\pi(f) - a}{b - a} \in [0,1].
\end{equation}
Now, it is enough to upper-bound the right-hand side of~\eqref{eq:target-after-convexity}. Denoting \( A := e^{\theta a} \) and \( B := e^{\theta b} \), it is enough to prove
\begin{equation}
\label{eq:target-after-convexity-simple}
(1 - \lambda) \left[ (1-\mu) \frac{A}{r - \lambda A} + \mu \frac{B}{r - \lambda B}\right]
\leq 1.
\end{equation}
By~\eqref{eq:r-admissible} we have \( r>\lambda\max\{A,B\} \). Put \( C := (1-\mu) A + \mu B \). Then,~\eqref{eq:target-after-convexity-simple} is equivalent to
\begin{equation}
(1 - \lambda) (r C - \lambda AB)
\leq (r - \lambda A) (r - \lambda B).
\end{equation}
Expanding brackets and collecting terms we get quadratic
\begin{equation}
r^2 - r\left[\lambda(A+B) + (1-\lambda) C\right]+ \lambda AB
\ge 0.
\end{equation}
Because \( r > 0 \),
\begin{equation}
\label{eq:r-quadratic-bound}
r + \frac{\lambda A B}{r}
\ge \lambda (A + B) + (1- \lambda) C
= (1+\lambda) \left[(1-p) A + p B\right],
\end{equation}
where
\begin{equation}
\label{eq:prob-p}
p
:= \frac{\lambda + (1-\lambda)\mu}{1 + \lambda} 
\in [0,1].
\end{equation}
Dividing by \( (1+\lambda) \),
\begin{equation}
\label{eq:r-to-bernoulli}
\frac{r + \lambda AB/r}{1 + \lambda}
\ge (1 - p) A + p B
\implies
(1 - \lambda) \pi\left( \frac{e^{\theta f}}{r - \lambda e^{\theta f}} \right)
\leq 1.
\end{equation}
The weighted AM--GM inequality gives
\begin{equation}
\label{eq:r-r-theta-pre}
\frac{r + \lambda AB/r}{1 + \lambda}
\ge r^{1/(1+\lambda)} \left( \frac{AB}{r} \right)^{\lambda/(1+\lambda)}
= \exp\left[ \frac{1-\lambda}{1+\lambda} \log r + \frac{\lambda}{1 + \lambda} \theta(a+b) \right].
\end{equation}
We now set \( r = R_\theta \) defined in~\eqref{eq:r-theta-def}. By Proposition~\ref{prop:cs-weights-positive}, the choice satisfies~\eqref{eq:r-admissible}. Substituting in~\eqref{eq:r-r-theta-pre} yields
\begin{equation}
\frac{R_\theta + \lambda AB/R_\theta}{1 + \lambda}
\ge \exp\Big[ \theta \underbrace{\frac{(1-\lambda) \pi(f) + \lambda(a+b)}{1 + \lambda }}_{(1-p) a + pb} + \frac{\theta^2 (b-a)^2}{8} \Big].
\end{equation}
On the other hand, the classical Hoeffding lemma for Bernoulli random variable \( X \) with \( \bbP(X = a) = (1-p) \) and \( \bbP(X =b) = p \) gives
\begin{equation}
\exp\Big[ \theta[(1-p) a + pb] + \frac{\theta^2 (b-a)^2}{8}\Big]
\ge (1 - p) e^{\theta a} + p e^{\theta b}
= (1 - p) A + p B.
\end{equation}
This proves~\eqref{eq:r-to-bernoulli} for \(r=R_\theta\), and hence~\eqref{eq:fan-to-prove}. Therefore, \(M_{e^{\theta f}}^{1/2}Q_\lambda M_{e^{\theta f}}^{1/2}\preceq R_\theta I\), or equivalently, \(\bnorm{M_{e^{\theta f}}^{1/2}Q_\lambda M_{e^{\theta f}}^{1/2}}\leq R_\theta\), as claimed.
\end{proof}

\begin{remark}[The effective Bernoulli parameter]
\label{rem:meaning-of-p}
The probability \( p \) in~\eqref{eq:prob-p} arises naturally and has a transparent interpretation. Combining the two fractions in~\eqref{eq:target-after-convexity-simple} and clearing the denominator produces a linear form \( c_a e^{\theta a} + c_b e^{\theta b} \) with \( c_a = 1 - (1-\lambda) \mu \) and \( c_b = \lambda + (1-\lambda) \mu \); recognizing it as a Bernoulli MGF requires normalizing, which gives \( p = c_b / (c_a + c_b) \), that is,~\eqref{eq:prob-p}. Equivalently,
\begin{equation}
\label{eq:p-shrinkage}
p - \frac{1}{2}
= \frac{1-\lambda}{1+\lambda} \left( \mu - \frac{1}{2} \right),
\end{equation}
so \( p \) is the endpoint weight \( \mu \) of the observable shrunk toward \( 1/2 \) by exactly the factor \( (1-\lambda)/(1+\lambda) \) that appears in the final bound. This is the whole effect of the Markov dependence at this step: the \( \lambda \)-term contributes symmetrically to the two endpoints and therefore pulls the effective Bernoulli parameter toward \( 1/2 \).
\end{remark}

\begin{proposition}
\label{prop:cs-weights-positive}
Under the assumptions of Lemma~\ref{lem:rank1-resolvent}, for every \(\theta \in \bbR\) and every \(x \in [a,b]\),
\[
R_\theta-\lambda e^{\theta x}>0,
\]
where \(R_\theta\) is defined in~\eqref{eq:r-theta-def}.
\end{proposition}
\begin{proof}
The case \(\lambda=0\) is immediate, so assume \(0<\lambda<1\). Fix \(x \in [a,b] \). For \(\theta \geq 0\), since \(x \leq b\) and \(b-\pi(f)\leq b-a\),
\begin{equation}
\log R_\theta-\log\lambda-\theta x
\geq \log R_\theta-\log\lambda-\theta b
\geq \frac{1+\lambda}{1-\lambda}\frac{(b-a)^2 \theta^2}{8}-(b-a) \theta -\log\lambda.
\end{equation}
The right-hand side is a quadratic in \(\theta\), whose minimum equals
\begin{equation}
-\frac{2(1-\lambda)}{1+\lambda}-\log\lambda
=: g(\lambda).
\end{equation}
Since
\begin{equation}
g'(\lambda)
= -\frac{(1-\lambda)^2}{\lambda(1+\lambda)^2}<0,
\qquad g(1)=0,
\end{equation}
we have \(g(\lambda)>0\) for \(0<\lambda<1\). For \(\theta<0\), the same argument, using \(x \geq a\) and \(\pi(f)-a \leq b-a\), gives the same lower bound.
\end{proof}

\subsection{The one-step problem is one-dimensional}

Lemma~\ref{lem:rank1-resolvent} is an inequality, and it is natural to ask how tight it is: its proof passes through a Cauchy--Schwarz bound and a convexity bound, either of which might in principle be loose. The operator being estimated is, by~\eqref{eq:ort-decomposition-mqm}, a rank-one perturbation of a multiplication operator, and the part of the spectrum of such an operator lying above the spectrum of the unperturbed part is governed by a single scalar equation. This scalar characterization makes precise the claim from the introduction that, after Lemma~\ref{lem:op-bound}, 
the remaining problem is essentially one-dimensional. We include the following proposition, which is not needed for Theorem~\ref{thm:scalar-hoeffding} but will be useful for the comparison in Section~\ref{sec:comparison}.

\begin{proposition}[Exact form of the one-step bound]
\label{prop:secular}
Let \( \lambda \in [0,1) \), let \( f \) be bounded and measurable, and write \( m := e^{\theta f} \) and \( \kappa := \lambda \norm{m}_{L^\infty(\pi)} \). For \( R > \kappa \) set
\begin{equation}
\label{eq:secular}
\Phi(R)
:= (1-\lambda) \, \pi\!\left( \frac{m}{R - \lambda m} \right).
\end{equation}
Then \( \Phi \) is finite, continuous and strictly decreasing on \( (\kappa, \infty) \) with \( \Phi(R) \to 0 \) as \( R \to \infty \), and
\begin{equation}
\label{eq:one-step-exact}
\bnorm{M_{m}^{1/2} Q_\lambda M_{m}^{1/2}}
= \inf \big\{ R > \kappa : \Phi(R) \leq 1 \big\}.
\end{equation}
In particular, if \( \Phi(R) > 1 \) for some \( R > \kappa \), then the right-hand side of~\eqref{eq:one-step-exact} is the unique root of the equation \( \Phi(R) = 1 \); otherwise it equals \( \kappa \).
\end{proposition}

\begin{proof}
Write \( u := m^{1/2} \) and \( B := M_m^{1/2} Q_\lambda M_m^{1/2} = \lambda M_m + (1-\lambda) \, u \otimes u \), by~\eqref{eq:ort-decomposition-mqm}. Note that \( B \) is self-adjoint and \( B \succeq 0 \), so \( \norm{B} = \sup \spec(B) \). For \( R > \kappa \) we have
\begin{equation}
0
< \frac{m}{R - \lambda m}
\leq \frac{\norm{m}_{L^\infty(\pi)}}{R - \kappa}
\quad \pi\text{-a.e.,}
\end{equation}
so \( \Phi(R) \) is finite; the integrand is strictly decreasing in \( R \) pointwise, and the stated continuity and limit follow by dominated convergence. The inequality \( \leq \) in~\eqref{eq:one-step-exact} is~\eqref{eq:fan-to-prove}: for every \( R > \kappa \) with \( \Phi(R) \leq 1 \) we have \( B \preceq R I \), hence \( \norm{B} \leq R \). For the converse, first \( B \succeq \lambda M_m \) gives \( \norm{B} \geq \norm{\lambda M_m} = \kappa \). Suppose now that \( \Phi(R) > 1 \) for some \( R > \kappa \); by the above there is then a unique \( R_\ast \in (\kappa,\infty) \) with \( \Phi(R_\ast) = 1 \), and the right-hand side of~\eqref{eq:one-step-exact} equals \( R_\ast \). Put
\begin{equation}
h := (R_\ast I - \lambda M_m)^{-1} u
= \frac{u}{R_\ast - \lambda m},
\end{equation}
which is bounded, hence in \( L^2(\pi) \), and nonzero. Since \( \la u, h \ra = \pi\big( m / (R_\ast - \lambda m) \big) = (1-\lambda)^{-1} \),
\begin{equation}
Bh
= \lambda M_m h + (1-\lambda) \la u, h \ra \, u
= \lambda M_m h + u
= \lambda M_m h + (R_\ast I - \lambda M_m) h
= R_\ast h .
\end{equation}
So \( R_\ast \in \operatorname{spec}(B) \) and \( \norm{B} \geq R_\ast \), which completes the proof.
\end{proof}

Lemma~\ref{lem:rank1-resolvent} is the assertion that \( R_\theta \) lies in the set appearing in~\eqref{eq:one-step-exact}, and its proof is the scalar convexity estimate that verifies this. Any other scalar exponential-moment estimate may be substituted at this point without modifying Lemma~\ref{lem:op-bound}. Note also that the Cauchy--Schwarz step in the proof of Lemma~\ref{lem:rank1-resolvent} is saturated precisely at the eigenvector \( h \) exhibited above, so no information is lost there.

\subsection{Comparison with existing proofs}
\label{sec:comparison}

The statements of the intermediate lemmas and the proposition above are essentially known. The novelty lies partly in their proofs and, more importantly, in the way these ingredients  are combined to obtain the proof of Theorem~\ref{thm:scalar-hoeffding}. Indeed,  Lemma~\ref{lem:op-bound} appears as~\citet[Lemma~6]{fan2021hoeffding}, itself a refinement of~\citet[Lemma~3.5]{miasojedow2014hoeffding}, and the role of operators of the form \( M_{e^{\theta f/2}} Q_\lambda M_{e^{\theta f/2}} \) has been long established: see the original argument of~\citet{leon2004optimal} in the Hoeffding setting and the spectral perturbation method of~\citet{lezaud1998chernoff} for Chernoff-type bounds.

\medskip
Proposition~\ref{prop:secular} is likewise closely related to existing work \citet{miasojedow2014hoeffding} derives the same equation \( \Phi(R)=1 \) in Eq.~3.8, uses the same eigenfunction as in the proof above, and, for finitely-valued observables, identifies the norm of the one-step operator with the root of this equation in Lemma~3.10(i). The difference is in how this equation is used. \citet{miasojedow2014hoeffding} needs the one-step norm to \emph{be} an eigenvalue, and~\eqref{eq:secular} need not have a root: for observables whose essential supremum is approached on sets of small measure, the left-hand side stays below \( 1 \) throughout \( (\kappa,\infty) \). Restoring a root is what forces the approximation of \( f \) by finitely valued observables, the accompanying positivity estimate, and the passage to the limit, which is identified in the work as the main source of difficulty in the general setting. Here that difficulty does not arise. Lemma~\ref{lem:rank1-resolvent} never asks for an eigenvalue: it uses only the implication~\eqref{eq:fan-to-prove}, which is a resolvent criterion and holds whether or not a root exists, and Proposition~\ref{prop:secular} disposes of the remaining case by recording that the infimum in~\eqref{eq:one-step-exact} is then \( \kappa \). No discretization is needed at any point.

\medskip
An inequality structurally identical to~\eqref{eq:r-quadratic-bound} appears in the proof of~\citet[Lemma~4, Eq.~12]{fan2021hoeffding}, where it bounds the leading eigenvalue of the reduced two-state chain operator. The proof of~\citet{fan2021hoeffding} retains the full strength of the Le{\'o}n--Perron reduction, which is a convex-order comparison between the additive functionals of two Markov chains, and reaches the scalar convexity step only at the end: Lemma~6 is followed by an asymptotic-CGF representation (Lemma~7, and with it simple-function approximation, essential-spectrum analysis and Weyl's theorem), a probabilistic two-state comparison (Lemma~9) and an explicit \( 2 \times 2 \) operator analysis (Lemma~10). The argument given here retains only Lemma~6 and bypasses Lemmas~7, 9 and~10. This is possible because Proposition~\ref{prop:secular} shows that after the operator-product reduction the problem is already one-dimensional, so~\eqref{eq:fan-to-prove} may be used directly as a resolvent criterion rather than interpreted as an eigenvalue equation, and ordinary scalar convexity applied to the single function~\eqref{eq:psi-r-def} suffices. The convex-order comparison is a strictly stronger statement than the operator bound requires.

\medskip
We conclude by noting that a short proof of a Hoeffding-type inequality for finite-state Markov chains was also given by~\citet{rao2019hoeffding}. Their route is different from both~\citet{fan2021hoeffding} and the present one: with \( \lambda := \norm{A - E_\pi}_{L^2(\pi) \to L^2(\pi)} \) and \( \bbE[f_i(Y_i)] = 0 \), \( \abs{f_i} \le a_i \), he obtains
\begin{equation}
\label{eq:intro-rao-bound}
\bbP_\pi\left\{ \Babs{ \sum_{i=1}^n f_i(Y_i) } \ge u \Big( \sum_{i=1}^n a_i^2 \Big)^{1/2} \right\}
\leq 2 \exp\left( - \frac{u^2 (1-\lambda)}{64 e} \right),
\end{equation}
by bounding even moments of the sum by expanding them into monomials, estimating each monomial by a product of powers of \( \lambda \), and concluding by Markov's inequality. This argument also avoids the two-state comparison, and does not require reversibility. But the resulting inequality is loose: in the normalization of~\eqref{eq:intro-fan-bound} the exponent is smaller than the sharp one by a factor of \( 32e/(1+\lambda) \). \citet{rao2019hoeffding} notes this, observing that the bound of~\citet{fan2021hoeffding} is sharper while their proof is ``arguably somewhat simpler''. The point of the present note is that the argument of Section~\ref{sec:scalar} is of comparable length, while it also yields the \emph{sharp} constant for time-dependent observables on a \emph{general} state space.

\section{Markov-dependent random matrices}
\label{sec:matrices}

Next we show that the same proof architecture yields a substantial simplification of the Hoeffding inequality for Markov-dependent random matrices presented in~\citet{neeman2024concentration}. In the existing argument, the matrix MGF is first reduced (via the multi-matrix Golden--Thompson inequality) to an operator problem on a lifted Hilbert space; the proof then proceeds through an asymptotic MGF representation, comparison with a two-state chain, and an explicit \(2 \times 2\) eigenvalue calculation. We show that all of these steps can be avoided. The key observation is that noncommutativity enters only in the initial Golden--Thompson reduction. Once this step is done, the resulting problem has exactly the same structure as in the scalar case.

\paragraph{Relation to recent work.}
Very recently, \citet{song2026matrix} combined the phase-sensitive product estimate of~\citet{neeman2024concentration} with the full-line multivariate trace inequality of~\citet{sutter2017multivariate} to obtain the sharp matrix Hoeffding bound, with dimensional prefactor \(d\) and the same dependence-adjusted exponent as in the sharp scalar inequality. The contribution of the present section is complementary: the two operator lemmas below give a direct proof of this phase-sensitive product estimate on the Hilbert--Schmidt space avoiding asymptotic spectral analysis and two-state comparison; the bounded-interval multi-matrix Golden--Thompson inequality then yields the theorem below, while its full-line counterpart gives the sharper bound of \citet{song2026matrix}.

\paragraph{Notation.}
Let \( \bbH_d(\bbC) := \left\{A\in\bbC^{d\times d}:A=A^\ast \right\} \) denote the space of self-adjoint \(d\times d\) complex matrices, and let \(\lambda_{\max}(A)\) denote the largest eigenvalue of \(A\in\bbH_d(\bbC)\). We write \(S_2^d:=\bbC^{d\times d}\) for the space of \(d\times d\) complex matrices, equipped with the Hilbert--Schmidt inner product \(\la A,B\ra_{\HS}:=\Tr(A^\ast B)\) and the associated norm \(\norm{A}_{\HS}:=\sqrt{\Tr(A^\ast A)}\). In finite dimensions, \(\norm{\cdot}_{\rm HS}\) coincides with the Frobenius norm. Let \( \cL(S_2^d) \) denote the space of bounded linear operators acting on \( S_2^d \).

\paragraph{Setup.}
As in the scalar case, let \((X_t)_{t\geq1}\) be a stationary Markov chain on a general measurable state space \(\cX\), with transition kernel \(P\) and stationary distribution \(\pi\). Suppose that \(P\) admits an absolute spectral gap \(1-\lambda>0\) on \(L^2(\pi)\). Let \( (F_t:\cX\to\bbH_d(\bbC))_{t\geq1} \) be a sequence of measurable matrix-valued functions. We work on the lifted Hilbert space \( \cH_d:=L^2(\pi;S_2^d) \). Lift the Markov operator to \(\cH_d\) by
\begin{equation}
(\wt P G)(x)
:= \int_{\cX} G(x')\,P(x,dx'),
\qquad G \in \cH_d,
\end{equation}
where the integral is understood in the Bochner sense. Let \(\Pi\) denote the orthogonal projection onto the constant \(S_2^d\)-valued functions, i.e., \( (\Pi G)(x) := \int_{\cX}G(x')\,\pi(dx') \). Set \( Q_\lambda:=\Pi+\lambda(I-\Pi) \). Note that \( \norm{\wt P - \Pi} = \norm{P - \Pi} = \lambda \): under the identification \( \cH_d \cong L^2(\pi) \otimes S_2^d \) one has \( \wt P - \Pi = (P - \Pi) \otimes \mathrm{id}_{S_2^d} \), where \( \Pi \) on the right denotes the projection of Section~\ref{sec:scalar} and \( \mathrm{id}_{S_2^d} \in \cL(S_2^d) \) is the identity, and \( \norm{A \otimes \mathrm{id}_{S_2^d}} = \norm{A} \) for every bounded \( A \). Let \( J:S_2^d \to \cH_d \) be the map acting as \( (JZ)(x):=Z \), so that \( \Pi=JJ^\ast \). All matrix-valued expectations below are understood as Bochner integrals.

\medskip \noindent
For a bounded measurable operator-valued map \( A:\cX\to\cL(S_2^d) \), we denote by \(M_A\) the corresponding multiplication operator on
\(\cH_d\), defined by
\begin{equation}
(M_A G)(x):=A(x)G(x),
\qquad G\in\cH_d.
\end{equation}
In particular, for \(t\geq1\) and \(\phi\in[-\pi/2,\pi/2]\), let \(T_{t,\phi}:\cX\to\cL(S_2^d)\) be defined by
\begin{equation}
T_{t,\phi}(x) \, Z
= \frac{\cos\phi}{2} \bigl[F_t(x) Z+ZF_t(x)\bigr],
\qquad Z\in S_2^d.
\end{equation}
Since \( F_t(x) \) is self-adjoint, so is \( T_{t,\phi}(x) \) with respect to \( \la \cdot, \cdot \ra_\HS \), and if \( F_t(x) \) has eigenvalues \( \mu_1, \dots, \mu_d \) with eigenvectors \( v_1, \dots, v_d \), then \( T_{t,\phi}(x) \) has eigenvalues \( \tfrac{\cos \phi}{2}(\mu_i + \mu_j) \) with eigenvectors \( v_i v_j^\ast \), \( 1 \le i, j \le d \). In particular, \( a_t I_d \preceq F_t(x) \preceq b_t I_d \) and \( \cos\phi \ge 0 \) give \( \cos\phi \, a_t I \preceq T_{t,\phi}(x) \preceq \cos\phi \, b_t I \) on \( S_2^d \). We write \( M_{t,\phi}(\theta) := M_{\exp \left(\theta T_{t,\phi}(x)\right)} \), where \(\exp\!\left(\theta T_{t,\phi}(x)\right)\) is defined by the functional calculus of the self-adjoint operator \(T_{t,\phi}(x)\). Let \(\bfI \in \cH_d\) denote the constant function \(\bfI(x)=I_d\), so that \( \norm{\bfI}_{\cH_d}^2 = \norm{I_d}_{\rm HS}^2 = d \).

\medskip
We can now state the main result of the section.

\begin{theorem}
\label{thm:matrix-hoeffding}
Assume that \( \pi(F_t)=0 \) and
\(a_tI_d\preceq F_t(x)\preceq b_tI_d\) for every \(t\geq1\) and \(x\in\cX\), where \(b_t>a_t\). Then, for every \(\varepsilon>0\),
\[
\bbP_\pi \left\{ \lambda_{\max} \left( \sum_{t=1}^n F_t(X_t) \right) \ge \varepsilon \right\}
\leq d^{2 - \pi/4} \exp\left( - \frac{1 - \lambda}{1 + \lambda} \cdot \frac{\varepsilon^2}{8 \sum_{t=1}^n (b_t - a_t)^2 / \pi^2} \right).
\]
\end{theorem}

\begin{remark}
\label{rem:centering}
Centering can be imposed without restricting the class of observables, but it comes at a cost. Applying Theorem~\ref{thm:matrix-hoeffding} to \( F_t - \pi(F_t) \) is legitimate, since \( a_t I_d \preceq F_t(x) \preceq b_t I_d \) forces \( a_t I_d \preceq \pi(F_t) \preceq b_t I_d \) and therefore \( -(b_t-a_t) I_d \preceq F_t(x) - \pi(F_t) \preceq (b_t-a_t) I_d \); the width doubles, so the exponent loses a factor \( 4 \). No such loss occurs in Theorem~\ref{thm:scalar-hoeffding}, because Lemma~\ref{lem:rank1-resolvent} carries \( \pi(f) \) explicitly.
\end{remark}

\paragraph{Discussion.} For real symmetric matrices, Theorem~\ref{thm:matrix-hoeffding} reproduces the Hoeffding bound in~\citet[Theorem~2.5]{neeman2024concentration}. For complex Hermitian matrices, it gives the same bound directly on the complex Hilbert space, without the additional factor \(2\) introduced by their realification argument in Corollary~2.7. More importantly, the proof follows exactly the same architecture as in the scalar case; see Section~\ref{sec:scalar}.

\begin{proof}
For \(\theta>0\) and \(\phi\in[-\pi/2,\pi/2]\), define the noncommutative MGF functional~\citep{neeman2024concentration}
\begin{equation}
\label{eq:matrix-mgf}
\cM_\phi(\theta)
:= \bbE_\pi \LRnorm{\prod_{t=1}^n A_{t,\phi}(X_t)}_{\HS}^2,
\qquad \text{where} \quad
A_{t,\phi}(x)
:= \exp\!\left(\frac{\theta e^{i\phi}}{2}F_t(x)\right).
\end{equation}
First, we use the bounded multi-matrix Golden--Thompson inequality~\citep[Theorem~3.4]{garg2018matrix} together with the Schatten estimate at exponent \(p=\pi/4\). More precisely, for every \(s>0\),
\begin{equation}
\bbE_\pi\!\left[\Tr\exp\!\left( \frac{\pi s}{4}\sum_{t=1}^n F_t(X_t) \right)\right]
\le d^{\,1-\pi/4} \int_{-\pi/2}^{\pi/2} \cM_\phi(s)\,\xi_n(d\phi).
\end{equation}
Here \(\xi_n\) is a fixed probability measure defined by the Golden--Thompson theorem; it may depend on \(n\), but not on \(s\), the matrices, or the realized sample path. Consequently, expectation and integration may be interchanged. Setting \(s=4\theta/\pi\) and applying Markov's inequality gives
\begin{equation}
\label{eq:matrix-gt}
\bbP_\pi\left\{ \lambda_{\max}\left(\sum_{t=1}^nF_t(X_t)\right) \geq \varepsilon \right\}
\leq d^{1-\pi/4} \inf_{\theta>0} e^{-\theta \varepsilon} \int_{-\pi/2}^{\pi/2} \cM_\phi\left(\frac{4\theta}{\pi}\right) \,\xi_n(d\phi).
\end{equation}
Second, Lemma~\ref{lem:matrix-op-bound} reduces the noncommutative MGF~\eqref{eq:matrix-mgf} to a product of one-step operator norms:
\begin{equation}
\label{eq:matrix-op-product}
\cM_\phi(\theta)
\leq d\prod_{t=1}^n \bnorm{ M_{t,\phi}(\theta)^{1/2} Q_\lambda M_{t,\phi}(\theta)^{1/2}}.
\end{equation}
Finally, Lemma~\ref{lem:matrix-one-step} gives, uniformly in \(\phi\in[-\pi/2,\pi/2]\),
\begin{equation}
\label{eq:matrix-one-step}
\bnorm{ M_{t,\phi}(\theta)^{1/2} Q_\lambda M_{t,\phi}(\theta)^{1/2}}
\leq \exp\left(
\frac{1+\lambda}{1-\lambda}
\frac{\theta^2(b_t-a_t)^2}{8}
\right).
\end{equation}
Combining \eqref{eq:matrix-op-product} and
\eqref{eq:matrix-one-step},
\begin{equation}
\label{eq:matrix-mgf-final}
\cM_\phi(\theta)
\leq
d\exp\left(
\frac{1+\lambda}{1-\lambda}
\frac{\theta^2}{8}
\sum_{t=1}^n(b_t-a_t)^2
\right).
\end{equation}
Substituting \eqref{eq:matrix-mgf-final} into
\eqref{eq:matrix-gt} yields
\begin{align}
\bbP_\pi\left\{ \lambda_{\max}\left(\sum_{t=1}^nF_t(X_t)\right) \geq\varepsilon \right\}
&\leq d^{2-\pi/4} \inf_{\theta>0} \exp\left( -\theta\varepsilon + \frac{1+\lambda}{1-\lambda} \frac{2\theta^2}{\pi^2} \sum_{t=1}^n(b_t-a_t)^2 \right) \\
& \leq d^{2-\pi/4} \exp\left( -\frac{1-\lambda}{1+\lambda} \frac{\pi^2\varepsilon^2} {8\sum_{t=1}^n(b_t-a_t)^2} \right),
\end{align}
as claimed.
\end{proof}

\medskip
We next prove the lemmas used above.

\begin{lemma}[Operator-product bound]
\label{lem:matrix-op-bound}
Under the setup of Theorem~\ref{thm:matrix-hoeffding}, for every \(\theta>0\) and \(\phi\in[-\pi/2,\pi/2]\),
\[
\cM_\phi(\theta)
\leq d\prod_{t=1}^n \bnorm{M_{t,\phi}(\theta)^{1/2} Q_\lambda M_{t,\phi}(\theta)^{1/2} }.
\]
\end{lemma}

\begin{proof}
Fix \( \theta \) and \( \phi \in [-\pi/2, \pi/2] \). Denote \( M_t := M_{t,\phi}(\theta) \) and \( C_t := M_{t}^{1/2} Q_\lambda M_{t}^{1/2} \). For every \(t \geq 1\), define the two-sided multiplication operator \( E_{t}:\cH_d \to \cH_d \) by
\begin{equation}
(E_{t} G)(x)
:= \exp\!\left(\frac{\theta e^{i\phi}}{4}F_t(x)\right) \, G(x) \, \exp\!\left(\frac{\theta e^{i\phi}}{4}F_t(x)\right)^\ast,
\qquad G \in \cH_d.
\end{equation}
Since \(F_t(x)\) is self-adjoint,
\(E_{t} \) is normal and \( E_{t}^\ast E_{t} = E_{t} E_{t}^\ast = M_{t} \). Since \( E_t \) is normal, this also means
\begin{equation}
\label{eq:op-norm-eqs}
\bnorm{E_{t} Q_\lambda E_t^\ast}
= \bnorm{Q_\lambda^{1/2} E_t^\ast E_{t} Q_\lambda^{1/2}}
= \bnorm{Q_\lambda^{1/2} M_{t} Q_\lambda^{1/2}}
= \bnorm{\underbrace{M_{t}^{1/2} Q_\lambda M_{t}^{1/2}}_{C_t}}.
\end{equation}

\medskip \noindent
As in~\eqref{eq:feynman-kac}, by the Markov property,
\begin{align}
\cM_\phi(\theta)
&= \left\la \bfI,\, E_{1}^2
\wt P E_{2}^2 \cdots \wt P E_{n}^2 \bfI \right\ra_{\cH_d}
= \left\la E_{1}^\ast\bfI,\, \prod_{t=1}^{n-1} \left( E_{t} \wt P E_{t+1} \right) E_{n}\bfI \right\ra_{\cH_d} \\
&\leq \norm{E_{1}^\ast\bfI}_{\cH_d} \norm{E_{n}\bfI}_{\cH_d} \prod_{t=1}^{n-1} \bnorm{ E_{t} \wt P E_{t+1} }. \label{eq:matrix-mgf-product}
\end{align}
As in the scalar case~\eqref{eq:scalar-k-def}, define \( K: \cH_d \to \cH_d \) by
\begin{equation}
K := \begin{cases}
\Pi + \lambda^{-1}(\wt P-\Pi), & \lambda>0,\\
I, & \lambda=0.
\end{cases}
\end{equation}
As in Lemma~\ref{lem:op-bound}, \( \norm{K} = 1 \) and \( \wt P=Q_\lambda^{1/2}KQ_\lambda^{1/2} \). Consequently, 
\begin{align}
\bnorm{E_{t}\wt P E_{t+1}}
&= \bnorm{E_{t} Q_\lambda^{1/2} K Q_\lambda^{1/2} E_{t+1}} \\
& \leq \bnorm{E_{t} Q_\lambda^{1/2}} \cdot \bnorm{Q_\lambda^{1/2} E_{t+1}}
= \underbrace{\bnorm{E_{t} Q_\lambda E_{t}^\ast}^{1/2}}_{= \norm{C_t}^{1/2}}
\underbrace{\bnorm{E_{t+1}^\ast Q_\lambda E_{t+1}}^{1/2}}_{= \norm{C_{t+1}}^{1/2}}.
\label{eq:matrix-cross}
\end{align}
Finally, since \(\bfI\) is constant, \( Q_\lambda^{1/2}\bfI=\bfI \). Therefore,
\begin{equation}
\norm{E_{1}^\ast\bfI}_{\cH_d}^2
= \left\la E_{1}^\ast Q_\lambda^{1/2}\bfI, E_{1}^\ast Q_\lambda^{1/2}\bfI \right\ra_{\cH_d}
= \left\la \bfI, Q_\lambda^{1/2} M_1 Q_\lambda^{1/2}\bfI \right\ra_{\cH_d}
\leq \underbrace{\norm{\bfI}_{\cH_d}^2}_{= d} \underbrace{\bnorm{ Q_\lambda^{1/2} M_1 Q_\lambda^{1/2} }}_{= \norm{C_1}}.
\end{equation}
Similarly, \( \norm{E_{n}\bfI}_{\cH_d} \leq \sqrt d\, \norm{ C_n }^{1/2} \). Substituting these estimates into \eqref{eq:matrix-mgf-product}, each one-step operator norm appears with total exponent one, and hence
\begin{equation}
\cM_\phi(\theta)
\leq d\prod_{t=1}^n \norm{ C_t }
= d\prod_{t=1}^n \bnorm{ M_{t,\phi}^{1/2} Q_\lambda M_{t,\phi}^{1/2} },
\end{equation}
as claimed.
\end{proof}

\paragraph{Discussion.}
The lemma above reformulates the operator reduction of~\citet{neeman2024concentration} directly in the Hilbert--Schmidt space of matrices. In~\citet{neeman2024concentration}, the noncommutative MGF is first represented on a lifted tensor-product space using Kronecker products, after which the original lifted Markov operator is replaced by the corresponding Le{\'o}n--Perron operator. Indeed, this is the purpose of the symmetrization argument in their Lemma~5.1 and the subsequent reductions in Section~5, culminating in the product bound used in their Hoeffding proof, cf.\ their Eq.~(6.10). Lemma~\ref{lem:matrix-op-bound} reaches the same product-of-one-step-norms representation directly. The only structural input is the factorization \( \wt P=Q_\lambda^{1/2}KQ_\lambda^{1/2} \) with \( \norm{K} = 1 \), while working intrinsically on \(\cH_d\) avoids the explicit Kronecker-product representation. Thus, at this stage, the matrix structure is carried entirely by the multiplication operators, whereas the dependence of the chain enters exactly through the same Le{\'o}n--Perron operator \(Q_\lambda\) as in the scalar case.

\medskip
This also makes the relation with the scalar argument particularly transparent. In the scalar proof, the Feynman--Kac representation is followed by the same contraction factorization to obtain a product of norms \( \norm{M_{e^{\theta f_t}}^{1/2}Q_\lambda M_{e^{\theta f_t}}^{1/2}} \). Here \(L^2(\pi)\) is simply replaced by \(\cH_d := L^2(\pi;S_2^d)\), and the scalar multiplication by \(e^{\theta f_t/2}\) is replaced by the natural two-sided multiplication on the Hilbert--Schmidt space of matrices. Unlike the scalar case, the half-step operators \(E_{t}\) arising from the noncommutative MGF are generally not self-adjoint: the complex phase forces \( E_{t}^2 \neq E_{t}^\ast E_{t} \). Consequently, splitting the Feynman--Kac product produces the two orientations \(E_{t} Q_\lambda E_{t}^\ast\) and \(E_{t}^\ast Q_\lambda E_{t}\), whereas in the scalar case they coincide automatically. The additional symmetrization step is therefore precisely what converts this genuinely noncommutative product structure into the positive one-step operator \(M_{t,\phi}^{1/2} Q_\lambda M_{t,\phi}^{1/2}\). Once its modulus \(M_{t,\phi}\) is formed, the remainder of the argument is identical: each transition operator is absorbed through \(\wt P=Q_\lambda^{1/2} K Q_\lambda^{1/2}\), leaving a product of positive one-step Le{\'o}n--Perron operator norms. The sole additional prefactor \(d\) comes from the endpoint vector, \( \norm{I_d}_{S_2^d}^2=d \), in place of \(\norm{\mathbf 1}_{\pi}^2=1\) in the scalar case.

\medskip
\begin{lemma}[One-step operator bound]
\label{lem:matrix-one-step}
Under the setup of Theorem~\ref{thm:matrix-hoeffding}, for every \(t\), \(\theta>0\), and
\(\phi\in[-\pi/2,\pi/2]\),
\[
\bnorm{M_{t,\phi}^{1/2} Q_\lambda M_{t,\phi}^{1/2}}
\leq \exp\left( \frac{1+\lambda}{1-\lambda} \frac{\theta^2(b_t-a_t)^2}{8} \right).
\]
\end{lemma}

\begin{proof}
Fix \(t\), \(\theta>0\), and \(\phi\in[-\pi/2,\pi/2]\), and write
\begin{equation}
c:=\cos \phi \in [0,1],
\qquad \bfT:=T_{t,\phi},
\qquad a:=c a_t,
\qquad b:=c b_t,
\qquad \mathbf M:=M_{e^{\theta \bfT}}.
\end{equation}
If \(c=0\), equivalently \( \phi = \pm \pi/2\), then \(\bfT=0\), \(\mathbf M=I\), and
\begin{equation}
\bnorm{\mathbf M^{1/2}Q_\lambda \mathbf M^{1/2}}
=\bnorm{Q_\lambda}
=1,
\end{equation}
so the claimed bound is immediate. Hence, in the remainder, assume \(c>0\), and therefore \(b-a=c(b_t-a_t)>0\).

\medskip
Since \(F_t(x)\) is self-adjoint and \( a_t I_d \preceq F_t(x) \preceq b_t I_d \), the operator \( \bfT(x)\) is self-adjoint on \(S_2^d\) and satisfies \( a I\preceq \bfT(x) \preceq b I \). Moreover, \(\pi(F_t)=0\) implies \( \pi(\bfT)=0 \). Hence
\begin{equation}
\label{eq:finite-rank-decom}
\mathbf M^{1/2} Q_\lambda \mathbf M^{1/2}
= \lambda \mathbf M+(1-\lambda)\mathbf M^{1/2} JJ^\ast \mathbf M^{1/2}.
\end{equation}
For \(r>\lambda e^{\theta b}\), set
\begin{equation}
D_r := rI - \lambda \mathbf M \succ 0.
\end{equation}
By the Schur-complement criterion,
\begin{equation}
\label{eq:op-iff}
\mathbf M^{1/2} Q_\lambda \mathbf M^{1/2}\preceq rI
\iff (1-\lambda) \mathbf M^{1/2} J J^\ast \mathbf M^{1/2} \preceq D_r
\iff (1-\lambda)J^\ast \mathbf M^{1/2} D_r^{-1}\mathbf M^{1/2} J \preceq I.
\end{equation}
Since \(\mathbf M\) and \(D_r\) are multiplication operators, \( \mathbf M^{1/2}\) and \( D_r \) commute; therefore
\begin{equation}
J^\ast \mathbf M^{1/2} D_r^{-1} \mathbf M^{1/2} J
= J^\ast \mathbf M D_r^{-1} J
= \pi\left[ \frac{e^{\theta \bfT}} {rI-\lambda e^{\theta \bfT}} \right]
= \pi\bigl[\psi_r(\bfT)\bigr],
\end{equation}
where \( \psi_r \) is given by~\eqref{eq:psi-r-def}, so that, via functional calculus for the self-adjoint operator \( \bfT(x) \),
\begin{equation}
\psi_r(\bfT(x))
= e^{\theta \bfT(x)} \big(r I - \lambda e^{\theta \bfT(x)}\big)^{-1}
\end{equation}
As in the scalar case, \(\psi_r\) is convex on \([a,b]\). Therefore the scalar inequality~\eqref{eq:psi-r-fx-convexity} lifts by functional calculus:
\begin{equation}
\psi_r(\bfT)
\preceq \frac{bI-\bfT}{b-a}\psi_r(a) + \frac{\bfT-aI}{b-a}\psi_r(b).
\end{equation}
Taking the expectation and using \(\pi(\bfT)=0\),
\begin{equation}
\pi[\psi_r(\bfT)]
\preceq \left[ \frac{b}{b-a}\psi_r(a) + \frac{-a}{b-a}\psi_r(b) \right]I.
\end{equation}
The quantity in brackets is exactly the same two-endpoint expression as in the scalar proof. Set
\begin{equation}
r
:= \exp\left( \frac{1+\lambda}{1-\lambda} \frac{\theta^2(b-a)^2}{8} \right).
\end{equation}
Integrating \(aI\preceq \bfT(x)\preceq bI\) and using \(\pi(\bfT)=0\) gives \(a \le 0 \le b\). Therefore, the scalar calculation in Proposition~\ref{prop:cs-weights-positive}, with mean zero, verifies \(r>\lambda e^{\theta b}\). The scalar argument gives \( (1-\lambda)\pi[\psi_r(\bfT)]\preceq I \). Consequently, by~\eqref{eq:op-iff}, \( \mathbf M^{1/2} Q_\lambda \mathbf M^{1/2} \preceq rI \) and therefore
\begin{equation}
\bnorm{M_{t,\phi}^{1/2}Q_\lambda M_{t,\phi}^{1/2}}
= \bnorm{\mathbf M^{1/2} Q_\lambda \mathbf M^{1/2}}
\leq r
= \exp\left( \frac{1+\lambda}{1-\lambda} \frac{\theta^2\cos^2\phi\,(b_t-a_t)^2}{8} \right).
\end{equation}
Finally, \(\cos^2\phi\leq1\), which yields the claim.
\end{proof}

This completes the proof of Theorem~\ref{thm:matrix-hoeffding}. We present separately a consequence of the fact that the argument was carried out on a complex Hilbert space throughout.

\begin{corollary}[Complex Hermitian observables]
\label{cor:complex-hermitian}
The bound of Theorem~\ref{thm:matrix-hoeffding} holds for complex Hermitian observables \( F_t: \cX \to \bbH_d(\bbC) \) with the same dimensional prefactor and the same exponent as in the real symmetric case \( F_t: \cX \to \bbH_d(\bbR) \).
\end{corollary}

\begin{proof}
Immediate from the proof of Theorem~\ref{thm:matrix-hoeffding}: the argument is carried out on \( \cH_d = L^2(\pi; S_2^d) \) with \( S_2^d = \bbC^{d \times d} \), and no step uses that the entries of \( F_t(x) \) are real. Self-adjointness of \( F_t(x) \), which is what Lemma~\ref{lem:matrix-op-bound} and Lemma~\ref{lem:matrix-one-step} require, holds by assumption.
\end{proof}

\paragraph{Discussion.}
Lemma~\ref{lem:matrix-one-step} is the direct Hilbert--Schmidt analogue of the one-step operator estimate used in the Hoeffding proof of~\citet{neeman2024concentration}; in their notation, the resulting bound corresponds to Eq.~(6.11). Their argument reaches this estimate after passing through the lifted tensor-product representation and the auxiliary comparison results of Sections~5--6. Here the same conclusion follows directly from the finite-rank decomposition~\eqref{eq:finite-rank-decom} and a Schur-complement resolvent criterion. The subsequent reduction is purely scalar: the function~\eqref{eq:psi-r-def} is convex, and the endpoint interpolation inequality lifts to the self-adjoint operator \(\bfT\) by functional calculus. In particular, no operator convexity is required. The connection with the scalar proof is therefore exact. There, \(Q_\lambda\) is a rank-one perturbation of \(\lambda I\), while here \(\Pi=JJ^\ast\) has rank \(\dim S_2^d=d^2\); apart from this finite-rank enlargement, the resolvent argument is unchanged. After taking expectation and using \(\pi(\bfT)=0\), the matrix inequality reduces to precisely the same two-endpoint scalar expression as before, so the remainder of the proof is inherited verbatim from the scalar Hoeffding argument.

\section{Conclusions}

In conclusion, the argument we presented rests on the two key facts: the contraction factorization \( P = Q_\lambda^{1/2} K Q_\lambda^{1/2} \) with \( \norm{K} = 1 \), which is independent of the observables, and a one-step resolvent bound which, by Proposition~\ref{prop:secular}, is a scalar problem. We conclude with two open problems. The first is to determine whether the right-spectral-endpoint refinement is valid in full generality for nonreversible chains and, if so, whether it can be recovered from the present factorization. The second is to establish whether replacing the Hoeffding lemma in Lemma~\ref{lem:rank1-resolvent} by a Bernstein- or Bennett-type estimate yields the corresponding Markov-chain inequalities, and to identify the resulting constants.

\section*{Acknowledgments}
I would like to thank Lorenzo Rosasco for proofreading the manuscript.

\medskip \noindent
I acknowledge financial support of the European Commission (Horizon Europe grant ELIAS~101120237). The European Commission and the other organizations are not responsible for any use that may be made of the information it contains.

\bibliography{references}
\bibliographystyle{apalike}

\end{document}